\documentclass[11pt,letterpaper]{article}
\usepackage{mathrsfs}
\usepackage{geometry}
\usepackage{graphicx}
\usepackage{amssymb}
\usepackage{epstopdf}
\usepackage{amsmath,amscd}
\usepackage{amsthm}
\theoremstyle{plain}
\newtheorem{theorem}{Theorem}

\newtheorem{lemma}[theorem]{Lemma}

\title{Uniqueness of Infinite Homogeneous Clusters in 1-2 Model}
\author{Zhongyang Li\footnote{Statistical Laboratory, Center for Mathematical Sciences,
Cambridge University, Wilberforce Road, Cambridge, CB3 0WB, UK, z.li@statslab.cam.ac.uk}}
\date{}
\begin{document}
\maketitle

\begin{abstract}
A 1-2 model configuration is a subset of edges of the hexagonal lattice such that each vertex is incident to one or two edges. We prove that for any translation-invariant Gibbs measure of 1-2 model, almost surely the infinite homogeneous cluster is unique.
\end{abstract}

\section{Introduction}A 1-2 model configuration is a choice of subset of edges of the hexagonal lattice such that each vertex is incident to one or two edges. An example of 1-2 model configurations is shown in Figure \ref{ot}.
\begin{figure}[htbp]\label{ot}
\centering
\scalebox{0.7}[0.7]{\includegraphics*{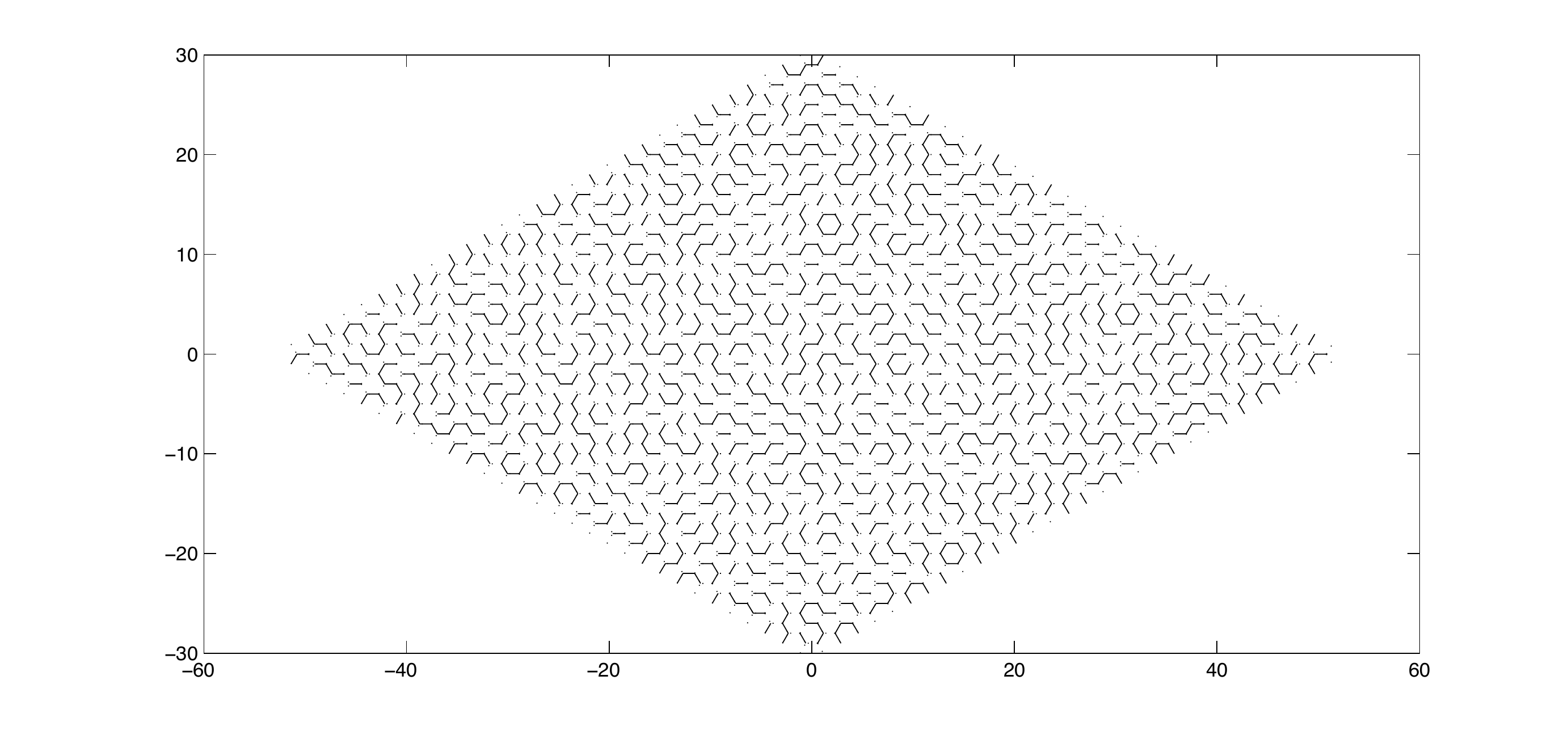}}
\caption{1-2 model configuration}
\end{figure}

The uniform 1-2 model (not-all-equal relation) was studied by computer scientists Schwartz and Bruck (\cite{sb}). They computed the partition function (total number of configurations) of the 1-2 model on a finite graph by the holographic algorithm (\cite{val}). In (\cite{li1}), we studied a generalized holographic algorithm, which could compute the local statistics of more general vertex models, including the non-uniform 1-2 model. A new approach to solve the 1-2 model is explored in (\cite{li2}) by constructing a measure-preserving correspondence between 1-2 model configurations on the hexagonal lattice and perfect matchings (\cite{k}) on a decorated lattice. Using a large torus to approximate the infinite planer graphs, we constructed in \cite{li2} an explicit translation-invariant, parameter-dependent measure for 1-2 model configurations on the planar hexagonal lattice, and proved that the 1-2 model percolates with respect to the changing parameters. See \cite{g} for an introduction of the percolation theory. In this paper, we prove that for any translation-invariant measure of 1-2 model configurations, almost surely there is at most one infinite homogeneous cluster.

Let $\mathbb{H}=(V,E)$ be the hexagonal lattice embedded into the whole plane. Let $v\in V$ be a vertex of $\mathbb{H}$. There is a one-to-one correspondence between configurations at $v$ (subsets of incident edges of $v$) and the set of all 3-digit binary numbers. Namely, since $v$ has 3 incident edges, we may assume that the horizontal edge of each vertex corresponds to the right digit, and a one-to-one correspondence between incident edges and digits can be constructed by moving counter-clockwise around a vertex and right-to-left along the digits. If an edge is included in the configuration, then the corresponding digit takes the value ``1''; otherwise the corresponding digit takes the value ``0''. See Figure \ref{loc} for examples of such a correspondence.
 \begin{figure}[htbp]\label{loc}
\centering
\includegraphics*{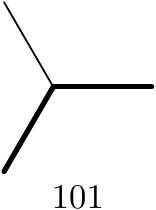}\qquad\qquad \includegraphics*{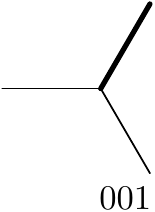}
\caption{local configurations and binary numbers}
\end{figure}

The weight function at a vertex is an assignment of a nonnegative number to each configuration at the vertex. For a 1-2 model configuration, since we require that each vertex can have only one or two incident edges, the weights for configurations $\{000\}$ and $\{111\}$ are 0. Moreover, throughout this paper we assume that the weight of the configurations at each vertex is 
\begin{eqnarray}
\begin{array}{cccccccc}000&001&010&011&100&101&110&111\\ 0&a&b&c&c&b&a&0\end{array},\label{sig}
\end{eqnarray}
where $a,b,c>0$ are arbitrary positive numbers. Given the weights of configurations as in (\ref{sig}), we say that an edge is $a$-type (resp. $b$-type, $c$-type) if it is the unique present edge in the configuration $\{001\}$ (resp. $\{010\}$, $\{100\}$). Given the correspondence of edges with the digits described previously, an edge is $a$-type if and only if it is horizontal; and starting from an $a$-type edge, moving counter-clockwise around a vertex, we meet the $b$-edge and the $c$-edge in cyclic order.

A Gibbs measure $\mu$ for the 1-2 model on $\mathbb{H}$ is a probability measure on the sample space of all possible 1-2 model configurations (denote the sample space by $\Omega$), such that for any finite subgraph $\Lambda\subset \mathbb{H}$, and any fixed configuration $\omega_{\Lambda^c}$ on the complement graph $\Lambda^c$, the probability of a configuration $\omega_{\Lambda}$ on $\Lambda$, conditional on $\omega_{\Lambda^c}$, is proportional to the product of configuration weights at each vertex of $\Lambda$. Namely,
\begin{eqnarray*}
\mu(\omega_{\Lambda}|\omega_{\Lambda^c})\propto\prod_{v\in \Lambda}w(\omega_{\Lambda}|_v)
\end{eqnarray*}
where $\omega_{\Lambda}|_v$ is the configuration of $\omega_{\Lambda}$ restricted at the vertex $v$, i.e. $\omega_{\Lambda}|_{v}$ is one of the six possible 1-2 model configurations $\{001\},\{010\},\{011\},\{100\},\{101\},\{110\}$, and $w(\cdot)$ is the weight function at a vertex.

A homogeneous cluster of a 1-2 model configuration, is a connected subset of vertices of $\mathbb{H}$, in which each vertex has the same configuration, i.e., one of $\{001\}$, $\{010\}$, $\{011\}$, $\{100\}$, $\{101\}$, $\{110\}$. See Figure \ref{101} for examples of homogeneous $\{101\}$ clusters.
 \begin{figure}[htbp]\label{101}
\centering
\scalebox{0.5}[0.5]{\includegraphics*{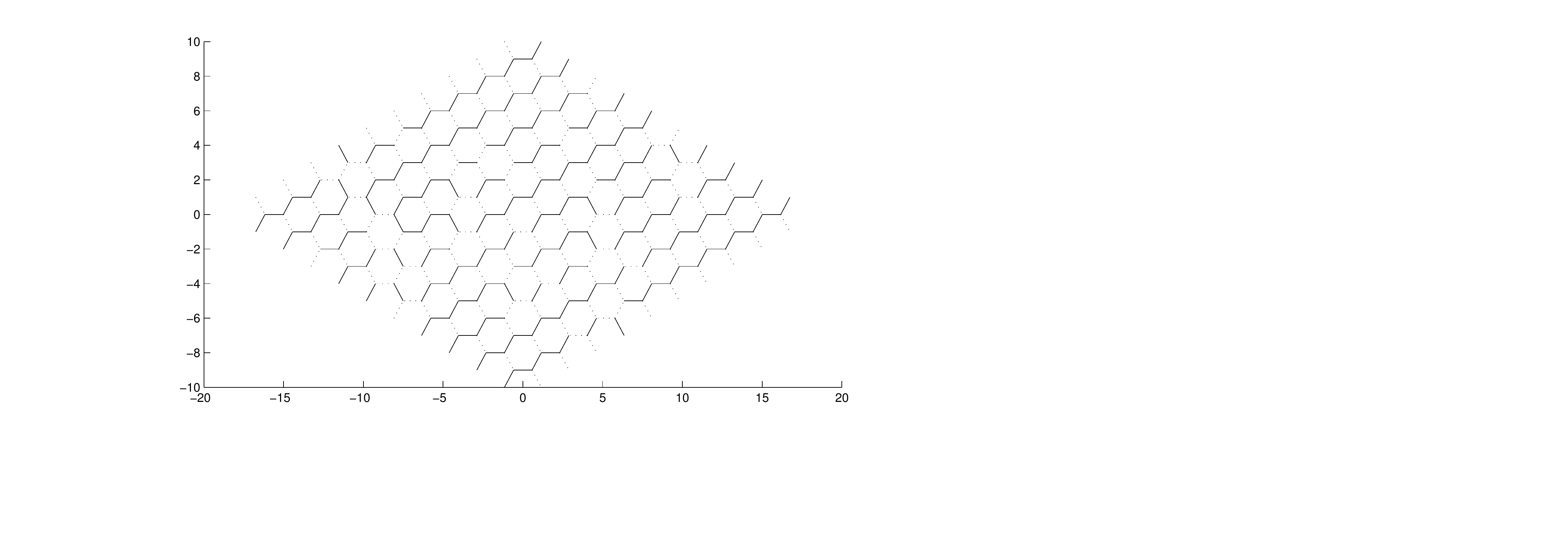}}
\caption{homogeneous $\{101\}$ clusters}
\end{figure}

 A homogenous cluster is infinite if it consists of infinitely many vertices. It is proved in \cite{li2} that under the explicitly constructed, parameter-dependent Gibbs measure, an infinite homogeneous cluster exists almost surely for some values of parameters, while infinite homogeneous clusters do not exist almost surely for some other values of parameters. The main theorem of this paper is as follows:

\begin{theorem}For any translation-invariant Gibbs measure of 1-2 model configurations on the whole-plane hexagonal lattice $\mathbb{H}$, almost surely there is at most one infinite homogeneous cluster.
\end{theorem}

It is proved by Burton and Keane (\cite{bk}) that for any translation-invariant, finite energy measure on $\{0,1\}^{\mathbb{Z}^d}$ configurations, almost surely there is at most one infinite open cluster. However, since the Gibbs measure for the 1-2 model does not satisfy the finite-energy condition in \cite{bk}, the proof in \cite{bk} does not work for the 1-2 model case. We prove the theorem for $\{001\}$ cluster in Section 2, and the result for all the other homogeneous clusters can be proved using exactly the same technique.

\section{Infinite Clusters}

\begin{lemma}\label{fno}Let $\mu$ be an ergodic, translation-invariant Gibbs measure for 1-2 model configurations. Let $\mathcal{N}_{\{001\}}$ be the number of infinite $\{001\}$-clusters. For any $1<k<\infty$,
\begin{eqnarray*}
\mu(\mathcal{N}_{\{001\}}=k)=0
\end{eqnarray*}
\end{lemma}
\begin{proof}Recall that in a $\{001\}$ configuration of a vertex, only the horizontal incident edge is present.  We prove the lemma by contradiction. Without loss of generality, assume $0<c\leq b\leq a$. Since $\mu$ is ergodic, and $\{\mathcal{N}_{\{001\}}=k\}$ is a translation-invariant event, then either $\mu(\mathcal{N}_{\{001\}}=k)=0$ or $\mu(\mathcal{N}_{\{001\}}=k)=1$. Assume there exists $1<k<\infty$, such that
\begin{eqnarray}\label{wa}
\mu(\mathcal{N}_{\{001\}}=k)=1.
\end{eqnarray}
Let $B_n$ be an $n\times n$ box of the hexagonal lattice centered at the origin. i.e, a rectangle domain with $n$ vertices incident to vertices outside the domain on each side, as illustrated in Figure \ref{hex}, in which the subgraph bounded by the outer dashed rhombus is a $5\times 5$ box, and the subgraph bounded by the inner dashed rhombus is a $3\times 3$ box.
\begin{figure}[htbp]\label{hex}
\centering
\scalebox{1}[1]{\includegraphics*{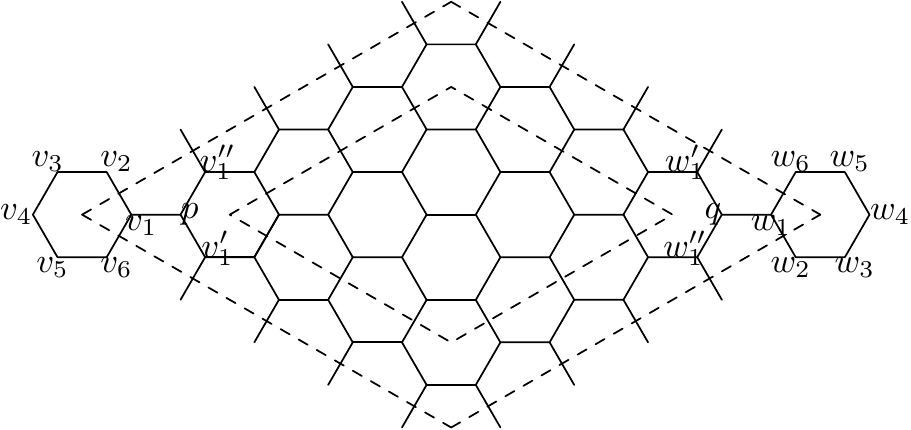}}
\caption{$n\times n$ box}
\end{figure}

Let $\mathcal{S}_n$ be the event that $B_n$ intersects all of the $k$ infinte $\{001\}$ clusters. Then
\begin{eqnarray*}
1=\mu(\cup_{n}\mathcal{S}_n)=\lim_{n\rightarrow\infty}\mu(\mathcal{S}_n).
\end{eqnarray*}
Hence there exists $N$, such that $\mu(\mathcal{S}_{N})>\frac{1}{2}$. Let $\omega$ be a configuration in $\mathcal{S}_{N}$. Consider the box $B_{N+2}$. There are two types of edges in $B_{N+2}$: either an interior edge whose both endpoints are in $B_{N+2}$; or a boundary edge connecting one vertex in $B_{N+2}$ and one vertex outside $B_{N+2}$. The boundary edges are those intersecting the outer dashed rhombus in Figure \ref{hex}. We are going to change the configurations in $B_{N+2}$ to derive a contradiction.

For the time being, we keep the configuration for all boundary edges; while for each interior edge of $B_{N+2}$, it is present if and only if it is horizontal (an $a$-type edge). There are three types of vertices in $B_{N+2}$: type I: a vertex whose all three neighbors are still in $B_{N+2}$; type II: a vertex with two neighbors in $B_{N+2}$, but one neighbor outside $B_{N+2}$; type III: a vertex with one neighbor in $B_{N+2}$, but two neighbors outside $B_{N+2}$. After the first step of changing configurations as described above, all type I vertices of $B_{N+2}$ has a configuration $\{001\}$. In particular, all the vertices of $B_{N}$ are type I vertices of $B_{N+2}$, hence all vertices of $B_{N}$ has a configuration $\{001\}$. All the type II vertices has at least one present edge and one unpresent edge, hence the configurations at type II vertices do not violate the rule that each vertex has degree 1 or 2. Now consider the type III vertices of $B_{N+2}$; there are only 2 type III vertices, lying in the two corners of $B_{N+2}$, labeled by $v_1$ and $w_1$ as in Figure \ref{hex}. They have at least one incident present edge, since the horizontal incident edge is present. We will describe the change of configurations around $v_1$ here; the change of configurations around $w_1$ are very similar. If $v_1$ has degree 1 or 2, then we are done and do not need to change configurations any more. If $v_1$ has degree 3, we consider the hexagon, denoted by $h_1$, which includes $v_1$ but does not include any other vertices of $B_{N+2}$. Let $v_1,v_2,v_3,v_4,v_4,v_5,v_6$ be all the vertices of $h_1$ in cyclic order, see Figure 2. If $v_3$ has configuration $\{001\}$, remove the edge $v_1v_2$, and we get a 1-2 model configuration. Similarly, if $v_5$ has configuration $\{001\}$, remove the edge $v_1v_6$, and we get a 1-2 model configuration. If neither $v_3$ nor $v_5$ have configuration $\{001\}$, we change the configuration as follows. Remove $v_1v_2$. After removing $v_1v_2$, if $v_2$ has degree 1, we are done. if $v_2$ has degree 0, add $v_2v_3$. After adding $v_2v_3$, if $v_3$ has degree 2, we are done. If $v_3$ has degree 3, remove $v_3v_4$. After removing $v_3v_4$, if $v_4$ has degree 1, we are done. If $v_4$ has degree 0, add $v_4v_5$. After adding $v_4v_5$, if $v_5$ has degree 2, we are done. If $v_5$ has degree 3, remove $v_5v_6$. Then $v_6$ has at least one present incident edge $v_1v_6$, and 1 unpresent incident edge $v_5v_6$. Hence $v_6$ has degree 1 or 2, we are done. Similar process applies for $w_1$ on the other corner.

Let $\omega'$ be the new configuration obtained from $\omega$ by the configuration-changing process described above.
We  $\omega'$ has exactly one infinite $\{001\}$. Note that if a vertex has $\{001\}$ configuration in $\omega$, then it has $\{001\}$ configuration in $\omega'$. Hence each infinite $\{001\}$ cluster of $\omega$ must be a subset of an infinite $\{001\}$ cluster of $\omega'$. Moreover, the configuration-changing process described above only changes configurations at a finite number of vertices, $\omega'$ cannot have infinite $\{001\}$ clusters which do not include an infinite cluster $\{001\}$ cluster of $\omega$.  Since all the infinite $\{001\}$ clusters of $\omega$ intersect $B_{N}$, all the infinite $\{001\}$ clusters of $\omega'$ intersect $B_{N}$. But all the vertices in $B_{N}$ has the configuration $\{001\}$. As a result, there is exactly one infinite $\{001\}$ cluster in $\omega'$. We consider the probability that such an configuration $\omega'$ occurs. This probability is bounded below by the probability of the event $\mathcal{S}_N$, multiplying a factor due to the change of configurations at finitely many vertices. Namely,
\begin{eqnarray*}
\mu(\mathcal{N}_{\{001\}}=1)>\frac{1}{2}\left(\frac{c}{6a}\right)^{2(N+2)^2+10}>0,
\end{eqnarray*}
which is a contradiction to (\ref{wa}), and the lemma follows.
\end{proof}

Before proving the theorem, we shall introduce the following notation. Let $Y$ be a finite set with at least three elements. A partition of $Y$ is a collection $P=\{P_1,P_2,P_3\}$ of the three non-empty disjoint subset of $Y$ whose union is $Y$. Partitions $P$ and $Q$ are compatible if there is an ordering of each such that $Q_2\cup Q_3\subset P_1$. A collection $\mathcal{P}$ of partitions is compatible if each pair $P,Q$ is compatible.

\begin{lemma}\label{cp}(\cite{bk})If $\mathcal{P}$ is a compatible partition of Y, then
\begin{eqnarray*}
|\mathcal{P}|\leq |Y|-2
\end{eqnarray*}
where $|\cdot|$ denote the cardinality of a set.
\end{lemma}

\begin{theorem}Let $\mu$ be a translation-invariant Gibbs measure on 1-2 model configurations $\Omega$. Then $\mu$-almost surely every $\omega\in\Omega$ has at most one infinite $\{001\}$-cluster.
\end{theorem}
\begin{proof}By ergodic decomposition, we may assume without loss of generality that $\mu$ is ergodic, so that $\mathcal{N}_{\{001\}}$, the total number of infinite $\{001\}$ clusters is constant $\mu$-almost surely. Then by Lemma \ref{fno}, $\mathcal{N}_{\{001\}}$ is either zero, one, or infinity. If $\mathcal{N}_{\{001\}}$ is zero or one, we are finished, so assume $\mathcal{N}_{\{001\}}$ is infinity. 

A box $B_n$ is an encounter box if the following two conditions hold
\begin{itemize}
\item There exists an infinite $\{001\}$ cluster $\mathcal{C}$, such that $B_n\subset \mathcal{C}$;
\item the set $\mathcal{C}\setminus B_{n+2}$ has no finite components and exactly three infinite components. 
\end{itemize}
We claim that if $\mu(\mathcal{N}_{\{001\}}=\infty)=1$, there exists $N$, such that the probability that the box $B_N$ centered at $(0,0)$ is an encounter box is strictly positive. To see that, let $B_{n}$ be an $n\times n$ box centered at the origin, as defined in the proof of Lemma \ref{fno}. Let $\mathcal{N}_{B_n,\{001\}}$ be the number of inifinite $\{001\}$ clusters intersecting $B_n$. Then
\begin{eqnarray*}
\lim_{n\rightarrow\infty}\mu(\mathcal{N}_{B_n,\{001\}}\geq 31)=1
\end{eqnarray*}
As a result, there exists $N$, such that
\begin{eqnarray*}
\mu(\mathcal{N}_{B_{N+2},\{001\}}\geq 31)>\frac{1}{2}
\end{eqnarray*}
Let $\omega$ be a 1-2 model configuration satisfying $\mathcal{N}_{B_{N+2},\{001\}}\geq 31$, then we can find three boundary vertices $u_1,u_2,u_3$ (vertices in $B_{N+2}$ with at least one neighbor outside $B_{N+2}$) of $B_{N+2}$, such that
\begin{itemize}
\item $u_1,u_2,u_3$ are in three different infinite $\{001\}$ clusters of $\omega$;
\item let $\mathcal{C}_1$, $\mathcal{C}_2$, $\mathcal{C}_3$ be the three different infinite clusters including $u_1$, $u_2$, $u_3$, then $(\mathcal{C}_1\cup\mathcal{C}_2\cup\mathcal{C}_3)\cap(\bar{h}_1\cup\bar{h}_2\cup\{v_1',v_1'',w_1',w_1''\})=\phi$, where $\bar{h}_1$ (resp. $\bar{h}_2$) consists of all the vertices in the hexagon $h_1$ (resp. $h_2$), as well as vertices incident to $h_1$ (resp. $h_2$), and $h_1$, $h_2$ are the two hexagons outside the two corners of $B_{N+2}$, as shown in Figure \ref{hex}.
\end{itemize}

Since $|\bar{h}_1\cup\bar{h}_2\cup\{v_1',v_1'',w_1',w_1''\}|=28$, the number of infinite $\{001\}$ clusters intersecting $\bar{h}_1\cup\bar{h}_2\cup\{v_1',v_1'',w_1',w_1''\}$ is at most 24. Moreover, since in $\omega$, the number of infinite $\{001\}$ clusters intersecting $B_{N+2}$ is at least 31, we can always find $u_1,u_2,u_3$, satisfying the conditions listed above.

To make $B_{N}$ an encounter box, we change configurations in $B_{N+2}$ as follows. First of all, we define the outer contour of $B_{N+2}$ to be the closed contour consisting of all the interior edges (edges connecting two vertices in $B_{N+2}$) sharing a  vertex with a boundary edge of $B_{N+2}$. Let all the horizontal edges (a-type edges) on the outer contour of $B_{N+2}$ be present. Let $u$ be a boundary vertex of $B_{N+2}$ (vertex with at least one neighbor outside $B_{N+2}$), other than $u_1,u_2,u_3,w_1,v_1$. $u$ is incident to three edges: $e_h$, the horizontal edge; $e_b$, the boundary edge; and $e_i$, the edge other than $e_h$ and $e_b$. $e_h$ is present for any such $u$ after the first step of changing configurations described above. We change the configurations of $e_i$, if necessary, such that if $e_b$ is present, then $e_i$ is not present; and if $e_b$ is not present, then $e_i$ is present. This way, all the boundary vertices of $B_{N+2}$ except $u_1,u_2,u_3,w_1,w_4$ have degree 2, and do not have a $\{001\}$ configuration. Let $p$ (resp. $q$) be the vertex adjacent to $v_1$ (resp. $w_1$) through a horizontal edge, see Figure \ref{hex}. The configurations of  edges $v_1p$ and $w_1q$ are rearranged according to the configurations of on the four other incident edges of $p$ and $q$, such that at vertices $p$ and $q$, the rule that one or two incident edges are present (1-2 law) is not violated. To make sure the configurations at $v_1$ and $w_1$ satisfy the 1-2 law, we will change the configurations on the edges of the hexagons $h_1$ and $h_2$ (see Figure \ref{hex}). We discuss the case of $h_1$ here, the case of $h_2$ is exactly the same.  We give $h_1$ a configuration such that each alternating side of $h_1$ is present. This way no matter what configurations are outside $h_1$, we always get a configuration on vertices of $h_1$ which does not violate the 1-2 law.

Obviously, after such a change of configurations, the only possible ways to connect $\{001\}$ clusters outside $B_{N+2}$ to $\{001\}$ clusters in $B_{N}$ is through vertices $u_1,u_2,u_3$. That is because any boundary vertices of $B_{N+2}$ except $u_1,u_2,u_3,v_1,w_1$ do not have a $\{001\}$ configuration. Moreover, the method to arrange configurations on $h_1$ (resp. $h_2$) implies that at least one non-horizontal edge incident to $v_1$ (resp. $w_1$) is present, hence $v_1$ (resp. $w_1$) does not have the $\{001\}$ configuration.   Now consider the box $B_N$ centered at the origin. Remove the boundary edges of $B_N$ from the configuration. For each interior edge of $B_N$, it is present if and only if it is horizontal. This way all the vertices in $B_N$ have a $\{001\}$-configuration. Moreover,  all the vertices in $B_{N+2}$ do not violate the 1-2 law. To check this claim, we only need to check the vertices of $B_{N+2}$ which are neighbors of vertices of $B_N$. Any such vertex has at least one present incident edges, namely the horizontal edge, and at least one non-present incident edges, namely the boundary edge of $B_N$, hence the 1-2 law is satisfied. 

Again Let $\omega'$ be the new configuration obtained from $\omega$ by the configuration-changing process described above.
It is trivial to check by definition that $B_N$ is an encounter box for the configuration $\omega'$. Consider the probability of those configurations in which $B_N$ is an encounter box, this probability is bounded below by the probability $\mu(\mathcal{N}_{B_{N+2},\{001\}}\geq 27)$, multiplying a factor caused by changing configurations at finitely many vertices. Namely,
\begin{eqnarray*}
\mu(B_N\ is\ an\ encounter\ box)\geq \frac{1}{2}\left(\frac{c}{6a}\right)^{2(N+2)^2+10}>0
\end{eqnarray*}

Let $B_{s(N+2)}$ be an $s(N+2)\times s(N+2)$ box centered at the origin, consisting of $s^2$ non-overlapping $(N+2)\times (N+2)$ boxes (each one is a translation of the other). Let $B_N(i,j)$ be the $N\times N$ box centered at $(i,j)$. Let $\mathcal{B}_N^s$ be the collection of all the $N\times N$ box included in one of the $s^2$ non-overlapping $(N+2)\times (N+2)$ boxes in $B_{s(N+2)}$, such that each $N\times N$ box and the corresponding $(N+2)\times (N+2)$ box have the same center.

By translation invariance, the probability that each $B_N(i,j)$  is an encounter box is at least $\frac{1}{2}\left(\frac{c}{6a}\right)^{2(N+2)^2+10}$, so the expected number of encounter boxes in $\mathcal{B}_N^{s}$ is at least
\begin{eqnarray}
\frac{1}{2}\left(\frac{c}{6a}\right)^{2(N+2)^2+10}s^2\label{lb}
\end{eqnarray}

Let $\mathcal{C}$ be a fixed infinite $\{001\}$ cluster of $\omega$. Define
\begin{eqnarray*}
Y=\mathcal{C}\cap\{outer\ boundary\ of\ B_{s(N+2)}\},
\end{eqnarray*}
where the outer boundary of $B_{s(N+2)}$ are the set of those vertices outside $B_{s(N+2)}$ and incident to  vertices in $B_{s(N+2)}$.

If $B_N(i,j)\in\mathcal{B}_{N}^s$ is an encounter box for $\omega$ with respect to $\mathcal{C}$, then the removal of $B_{N+2}(i,j)$ from $\mathcal{C}$ defines a partition
\begin{eqnarray*}
P=\{P_1,P_2,P_3\}
\end{eqnarray*}
of the set $Y$, such that $P_i\neq \phi$ for $1\leq i\leq 3$. Namely, if $\mathcal{D}_1$, $\mathcal{D}_2$, $\mathcal{D}_3$ be the three components of $\mathcal{C}\setminus B_{N+2}(i,j)$, then let $P_i=\mathcal{D}_i\cap Y$. 

Moreover, if $B_N(i',j')\in\mathcal{B}_N^s$ is another encounter box with respect to the same infinite $\{001\}$ cluster $\mathcal{C}$ such that $(i',j')\neq (i,j)$ then $B_N(i',j')$ gives another partition $Q=\{Q_1,Q_2,Q_3\}$ of $Y$, and the indices of $P$ and $Q$ can be chosen in such a way that $Q_2\cup Q_3\subset P_1$; simply choose $Q_1$ to correspond to the component of $\mathcal{C}\setminus B_{N+2}(i',j')$ containing $B_N(i,j)$. Hence the set of partitions corresponding to encounter boxes of $\mathcal{C}$ in $\mathcal{B}_N^s$ forms a compatible partition of $Y$. By Lemma \ref{cp}, the number of compatible partitions is at most $|Y|-2$, where $|Y|$ is the number of vertices in $Y$. Summing over all different infinite clusters, we have the total number of encounter boxes in $\mathcal{B}_N^s$ is bounded above by $4s(N+4)$, which is less than (\ref{lb}) when $s$ is large. The contradiction shows that it is $\mu$-a.s. impossible that there are infinite many infinite ${\{001\}}$ clusters. By Lemma \ref{fno} almost surely there is at most one infinite $\{001\}$ cluster.
\end{proof}

\noindent\textbf{Acknowledgement} This work was supported by the Engineering and Physical Sciences Research Council under grant EP/103372X/1. The author thanks Geoffrey Grimmett for suggesting the possibility of this topic.


\begin{thebibliography}{99}
\bibitem{bk}R.~M.~Burton and M.~Keane,Density and Uniqueness in Percolation, Commun. Math. Phys. 121(1989), 501-505 
\bibitem{g}G.~Grimmett, Percolation, Springer-Verlag, Berlin (2003).
\bibitem{k}R.~Kenyon, Local Statistics on Lattice Dimers, Ann.
Inst. H. Poincar$\acute{e}$. Probabilit$\acute{e}$s, 33(1997),
591-618
\bibitem{li1} Z.~Li, Local Statistics of Realizable Vertex Models, Commun.~Math.~Phys.~304,723-763 (2011)
\bibitem{li2} Z.~Li, 1-2 Model, Dimers and Clusters
\bibitem{ns}C.~M.~Newman, L.~S.~ Schulman: Infinite clusters in percolation models. J. Stat. Phys. 26(1981),
613-628
\bibitem{sb}M.~Schwartz and J.~Bruck, Constrained Codes as Network of
Relations, Information Theory, IEEE Transactions, 54(2008), Issue 5,
2179-2195
\bibitem{val}L.~G.~Valiant, Holographic Algorithms(Extended Abstract),
in Proc. 45th IEEE Symposium on Foundations of Computer
Science(2004), 306-315
\end{thebibliography}
\end{document}